\documentclass[12pt]{amsart}
\usepackage{amsthm}
\usepackage{graphicx, psfrag, epstopdf} 
\usepackage{amssymb,amsmath,amscd,amsthm,verbatim}
\usepackage{amsfonts}
\usepackage{mathrsfs}

\newcommand{\norm}[1]{\bigl\| #1 \bigr\|}
\newtheorem{theo}{Theorem}
\newtheorem{define}[theo]{Definition}

\newtheorem{lemma}{Lemma}[section]
\newtheorem{coro}[lemma]{Corollary}
\newtheorem{prop}[lemma]{Proposition}
\newtheorem{conjecture}{Conjecture}
\theoremstyle{definition}
\newtheorem*{remark}{Remark}

\def\bbK{{\mathbb{K}}}

\def\cB{\mathcal B}

\newcommand{\Q}{{\mathbb Q}}

\newcommand{\T}{{\mathbb T}}
\newcommand{\R}{{\mathbb R}}
\newcommand{\N}{{\mathbb N}}

 \newcommand{\Z}{{\mathbb Z}} 
 \newcommand{\cZ}{{\mathcal Z}} 
 \newcommand{\cC}{{\mathcal C}} 

\def\tl{\tilde{l}}

\def\tp{\tilde{p}}
\def\tD{\tilde{\Delta}}

\def\th{{\theta}}
\def\E{{\mathbb{E}}}
\def\Prob{{\mathbb{P}}}

\def\eps{{\varepsilon}}

\def\bC{{\mathbf{C}}}
\def\bD{{\mathbf{D}}}

\def\bm{{\mathbf{m}}}

\def\brC{{\bar C}} 
\def\brsD{{\bar\sD}} 

\def\brm{{\bar m}}
\def\brR{{\bar R}} 
\def\brrR{{\bar \brR}} 
 
\def\brv{{\bar v}} 
\def\brX{{\bar X}} 
\def\brZ{{\bar Z}} 
\def\brrZ{{\bar \brZ}} 
\def\bralpha{{\bar\alpha}}
\def\brlambda{{\bar\lambda}}

\def\bbD{{\mathbb{D}}}

\def\th{\underline{\theta}}

\def\cO{\mathcal O}

\def\cD{\mathcal D}
\def\cL{\mathcal L}
\def\cM{{\mathcal{M}}}

\def\sD{\mathfrak{D}}
\def\sL{\mathfrak{L}}
\def\brsL{\bar{\sL}}
\def\sP{\mathfrak{P}}

\def\RmII{{I\!\!I}}

\def\Cyl{{\mathbb{C}}}
\def\Const{{\rm Const}}
\def\mes{{\rm mes}}

\def\Vol{{\rm Vol}}
\def\N{{\mathbb{N}}}

\def\reals{{\mathbb{R}}}

\def\tf{{\tilde f}}

\def\a{\alpha}
\def\th{\theta}

\def\e{\epsilon}
\def\l{\lambda}

\def\Sp{{\mathbb{S}}}

\author{Dmitry Dolgopyat and Bassam Fayad}
\title[Ergodic sums for Toral Translations]{Deviations of Ergodic sums for Toral Translations  \\  I. Convex bodies}

\address{Dmitry Dolgopyat: Department of Mathematics and Institute for Physical Science and Technology University of Maryland College Park MD 20742 USA}
\email{dmitry@math.umd.edu}

\address{Bassam Fayad: 	Institut de MathŽmatiques de Jussieu, CNRS U.M.R. 7586, 175 rue du Chevaleret  75013, Paris 
France }
\email{bassam@math.jussieu.fr}
\begin{document}
\maketitle

\begin{abstract} We show the existence of a limiting distribution $\cD_\cC$ of the adequately normalized discrepancy function of a random translation on a torus relative to a strictly convex set $\cC$.  Using a correspondence between the small divisors in the Fourier series of the discrepancy function and lattices with short vectors,  and mixing of diagonal flows on the space of lattices, we identify $\cD_\cC$ with   the distribution of the level sets of a function defined on the product of the space of lattices with an infinite dimensional torus. 
We apply our results to counting lattice points in slanted cylinders and to time spent in a given
ball by a random geodesic on the flat torus.
\end{abstract}

\section{Introduction}
One of the surprising discoveries of dynamical systems theory is that many deterministic systems with
non-zero Lyapunov exponents satisfy the same limit theorems as the sums of independent random variables.
Much less is known for the zero exponent case where only a few examples have been analyzed
(\cite{B, BF, FF, M0}). In this paper we consider the extreme case of toral translations where the map 
not only has zero exponents but is actually an isometry. In this case it is well known that ergodic sums of
smooth observables are coboundaries and hence bounded for almost all translation vectors, so we consider
the case where the observables are not smooth, namely, they are indicator functions of nice sets 
(another possibility is to consider meromorphic functions, cf. \cite{HL, SU}). 
The case of circle rotations was studied by Kesten \cite{K1, K2} who proved the following result

\begin{theo}
Let $0<r<1$, and let 
$$D_N(r,x, \alpha)=\sum_{n=0}^{N-1} \chi_{[0,r]}(x+n\alpha)-N r.$$ 
There is a number
$\rho=\rho(r)$ such that if $(x,\alpha)$ is uniformly distributed on $\T^2$ then
$\frac{D_N}{\rho \ln N}$ converges to a standard Cauchy distribution, that is,
$$ \mes\left((x, \alpha): \frac{D_N(r,x,\a)}{\rho \ln N}\leq z\right)\to \frac{\tan^{-1} z}{\pi} +\frac{1}{2}. $$
Moreover $\rho(r)\equiv \rho_0$ is independent of $r$ if $r \not\in \Q$ and 
it has a non-trivial dependence on $r$ if $r \in \Q.$
\end{theo}

Our goal is to extend this result to higher dimensions. An immediate question is what kind of sets one
wants to consider in the definition of discrepancies. There are  two natural counterparts to intervals in higher dimension: balls and boxes. In this paper we will deal with balls and more generally with strictly convex and analytic bodies $\mathcal{C}$. Given a convex body $\cC$, we consider the family $\mathcal{C}_r$ of convex bodies obtained from 
$\mathcal C$ by rescaling it with a ratio $r>0$ (we apply to $\cC$ the homothety centered at the origin with scale $r$).  We suppose $r<r_0$ so that the rescaled bodies can fit inside the unit cube of $\R^d$. We define
\begin{equation}
D_\cC(r,\a,x,N)  = \sum_{n=0}^{N-1} \chi_{\cC_r}(x+n\a) - N \Vol({\cC_r})
\end{equation}
where 
$\chi_{\cC}$ is the indicator function of the set $\cC.$

We will assume that $(r,\a,x)$ are uniformly distributed in $X=  [a,b] \times \T^d\times \T^d$ and denote by $\lambda$ the normalized Lebesgue measure on $X$. Then we will prove the following 
\begin{theo} \label{main.limit}  $ $ Let  $\mathcal{C}$ be a strictly convex analytic body that fits inside the unit cube of $\R^d$.There exists a distribution function $\cD_\cC(z) : \R \to [0,1]$ such that for any $b>a>0$,  we have 
\begin{equation} \label{limit2} \lim_{N \to \infty} \lambda \{ (r,\a,x) \in [a,b] \times \T^d\times \T^d 
\ \Big| \ \frac{D_\cC(r,\a,x,N)}{r^{\frac{d-1}{2}} N^{\frac{d-1}{2d}} } \leq z\}  =\cD_{\mathcal C}(z).
\end{equation}  
\end{theo} 
The explicit form of $\cD_\cC$ will be given in Proposition \ref{PrNS} of Section \ref{s1}.

\begin{remark}
The assumption that $r$ is random in Theorem \ref{main.limit} is needed to suppress possible irregular dependence 
of the limiting distribution on $r.$ We know from the work of Kesten that for $d=1$ the statement  becomes 
more complicated if $r$ is fixed. However it is likely that for $d\geq 2$ the limiting distribution is the same for
all $r.$ (Note that the limiting distribution $\cD_{\mathcal C}(z)$ is independent of the interval where $r$ is varying,
to achieve this we need to divide the LHS of \eqref{limit2} by $r^{\frac{d-1}{2}}.$)
\end{remark}

\begin{remark} The theorems in this paper are stated for $r,x,\a$ distributed according to Lebesgue measure, but it appears clearly from the proofs that the same results hold for any measure with smooth density with respect to Lebesgue. 
\end{remark}

\begin{remark} It is possible to consider different scaling regimes in the discrepancy function, by replacing $r$ with $r N^{-\gamma}$. For $\gamma>1/d$,    then the set of orbits of size $N$ which visit $\cC_{r N^{-\gamma}}$ at least once has small measure if $N$ is large.  The case $\gamma=1/d$, which is often coined as the Poisson regime,  was treated by 
Marklof in \cite{M-ETDS}, where he showed that the number of visits to $\cC_{N^{-1/d}}$ has a limiting distribution
without a need for normalization. (We also note that \cite{MS1} obtained Poisson regime versions of our Theorems \ref{ThCont} and \ref{ThCyl}). 
We will see in Section \ref{sec61} that for any $\gamma < 1/d$,  Theorem \ref{main.limit} still holds with the same limit distribution (with the normalization  $r^{\frac{d-1}{2}}N^{ \frac{d-1}{2d} (1-\gamma d) }).$
\end{remark}

Moreover, in the study of discrepancies in higher dimension it is possible to consider continuous time translations. Namely, 
let
\begin{equation}
\bD_\cC(r,v,x,T)  = \int_{0}^{T} \chi_{\cC_r}(S_v^t x) dt - T \Vol({\cC_r})
\end{equation}
where $S^t_v$ denotes the translation flow on the torus 
$\T^d=\R^d/\Z^d$, $d\geq 2$, with constant vector field given by the vector $v=(v_1,\ldots,v_d) \in \R^d$. 
We denote $\bD_\cC(v,x,T)  = \bD_\cC(1, v, x, T) .$

We suppose that $v$ is chosen according
to a smooth density $p$ whose support is compact and does not contain the origin. Let $\bar \sigma$ denote 
the product of the distribution of $v$ with the Haar measure on $\T^d, $ while $\sigma$ denotes the product of the normalized Lebesgue measure on $[a,b]$ with $\bar \sigma$.   In the case of dimension $d=2$, we will not need to consider a random scaling factor $r$ of the convex body and we will have that the distribution 
$\bD_{\cC}(v,x,T)$ converges without any normalization to some limit. 

\begin{theo} \label{ThCont} $ $ Let  $\mathcal{C}$ be a strictly convex analytic body that fits inside the unit cube of $\R^d$.

(a) If $d=2$, there exists a two-parameter family of 
distribution functions $\brsD_{\cC, v} (z) : \R \to [0,1]$, such that   for any $b>a>0$,  we have  
\begin{equation*}
\lim_{T\to\infty} \bar \sigma((v,x) \Big| \bD_\cC(v, x, T)\leq z)=
\int  \brsD_{\cC, v} (z) p(v) dv  
\end{equation*}

(b) If $d\geq 4$, there exists a $d$ parameter family of 
distribution functions $\sD_{\cC, v} (z) : \R \to [0,1]$ such that for any $b>a>0$,  we have 
\begin{equation} \label{limit.continuous} \lim_{T \to \infty}  
\sigma \{ (r,v,x)  \Big| \ \frac{\bD_\cC(r,v,x,T)}{ r^{\frac{d-1}{2}} T^{\frac{d-3}{2(d-1)}} } \leq z\}  
=\int \sD_{\cC,v}(z) p(v) dv.  
\end{equation}
\end{theo} 
The explicit forms of $ \brsD_{\cC, v}$ and  $\sD_{\cC, v}$ will be given in Propositions 
\ref{prop.theorem3a} and \ref{prop.theorem3b} 
 of Section  \ref{sectflow}

We show in Theorem \ref{th8} that 
in the case of balls, the limit distribution of the flow discrepancy 
does not depend on the 
distribution of the direction of the vector field. 

The case $d=3$ is different and cannot be treated with the same approach we use here.
In \cite{DF2} we prove that for $d=3$, $\frac{\bD_{\mathcal C}(r,v,x,T)}{r\ln T}$ converges to a Cauchy distribution
as $T\to\infty.$

\begin{remark}
We note that in Theorems \ref{main.limit} and \ref{ThCont}, the same limit holds if we consider translated sets of $T_u \cC_r$ since this amounts to replacing $x$ by $x-u$.
Also our results remain valid for tori of the form $\R^d/L$ where $L$ is an arbitrary lattice in $\R^d$ since by a linear
change of coordinates we can  reduce the problem to the case $L=\Z^d.$
\end{remark}

Before we go to the next section where we describe the limiting distribution $\cD_\cC$, let us observe that the least restrictive requirement on the set seems to be that $\cC$ is {\it semialgebraic}, 
that is it is defined by a finite number of algebraic inequalities. This would allow a diverse collection of
sets including balls, cubes, cylinders, simplexes etc.
\begin{conjecture} 
If $\cC$ is semialgebraic then there is a sequence $a_N=a_N(\cC)$ such that for a random translation
of a random torus $D_N/a_N$ has a limiting distribution. Here
$$ D_N(x, \alpha, L)=\sum_{k=0}^{N-1} \chi_\cC(x_k)-N \frac{\Vol(\cC)}{\text{\rm covol}(L)} $$
where $x_k=x+k\alpha \mod L,$ $L=A\Z^d$ and we assume that the triple 
$(A, x, \alpha) \in {\rm GL}(d,\R)/{\rm SL}(d,\Z)\times \T^d \times \T^d$ has a smooth compactly supported density (with respect to the product of the Haar measures).
\end{conjecture}
Note that there are two equivalent points of view. Either we fix $\cC$ and change the torus $\T^d=\R^d/L$
or we can fix the torus $\T^d=\R^d/\Z^d$ and change the set $\cC_A=A^{-1} \cC.$ As before, we 
introduced parameters into this problem to avoid an irregular behavior of the limiting distribution on the
set $\cC$  which appears in Kesten's result.

In a forthcoming paper \cite{DF2} we verify this conjecture for boxes. 
In that case we get a 
result similar to Kesten's, namely that 
$D_N/\ln^d N$ converges to Cauchy distribution.
We note that the study of discrepancy for boxes has a long history,
see \cite{beck} and references therein.


We note that the fact that ergodic sums of smooth observables are almost surely coboundaries is the starting point
of perturbation theories for nearly integrable conservative systems. Namely for smooth perturbations, 
adiabatic invariants diffuse very slowly (Nekhoroshev theory) and the diffusion takes place on a
set of very small measure (KAM theory). A completely different behavior emerges if we consider 
piecewise smooth perturbations \cite{dSD, DF1, GK, H} but non smooth perturbations are much less
studied than the smooth ones. From this point of view our paper can be regarded as a study of the 
 diffusion speed in the simplest skew product system
\begin{equation}
 I_{n+1}=I_n+\eps A(x_n), \quad x_{n+1}=x_n+\alpha 
\end{equation}
where $A(x)=\chi_\cC(x)$ 
We hope that the results of this paper can be useful in the study of a wider class of fully coupled
perturbations such as 
\begin{equation}
 I_{n+1}=I_n+\eps A(x_n, I_n), \quad x_{n+1}=x_n+\alpha(I_n)+\eps\beta(x_n, I_n) ,
\end{equation}
but this will be a subject of a future investigation.

Another potential application of our result is to deterministic (quasi-periodic) random walks. 
In this problem (see \cite{CC} and references therein) one considers a map $A:\T^d\to\Z^q$
of zero mean and asks if the random walk
$ S_N=\sum_{n=0}^{N-1} A(x+n\a) $ returns to a given set $K$ infinitely many or only finitely many time.
The first step in the study of such problems is to find a sequence $a_N$ such that $S_N/a_N$ has a non
trivial limiting distribution. If such $a_N$ is found then assuming that $S_N$ is more or less uniformly 
distributed in the ball of radius $a_N$ we have that $\mathbb{P}(S_N\in K)$ is of order $a_N^{-q}.$
One then expects that $S_N$ visits $K$ infinitely often if and only if $\sum_N a_N^{-q}=+\infty.$ Thus while our
results are not immediately applicable to deterministic random walks they allow to make plausible conjectures
about the values of $d$ and $q$ for which  the walk is recurrent.

While the motivations mentioned above will be subject of future investigations, we provide in Section \ref{ScExt}  two, more straightforward, applications of our results. One (subsection \ref{SSSlanted})
deals with number theory
(counting lattice points in slanted cylinders) and the other 
(subsection \ref{SSRandGeo}) 
deals with geometry (measuring the
time a random geodesics spends in a ball).
\smallskip

\noindent {\bf Plan of the paper. } The rest of the paper is organized as follows. In Section \ref{s1} we provide formulas for the limiting
distributions in Theorems \ref{main.limit}  and \ref{ThCont}.
In Sections \ref{ScNonRes}--\ref{sec.oscill} we prove Theorem \ref{main.limit}.
The proof consists of three parts. In Section \ref{ScNonRes} we consider the Fourier transform of the discrepancy function
and show that the main contribution comes from a small number of resonant terms. The computations here are close to the
one-dimensional computations done in \cite{K1}.
In Section \ref{ScGeom} we use the Dani correspondence (\cite{Da}) to relate the structure of the resonances to the dynamics of homogeneous flows on the  
space of lattices in $\R^{d+1}.$ 
Namely, an approach inspired by the work of Marklof (see \cite{M-ETDS, M2}) allows us to express the limiting distribution of resonances in terms of the distribution of 
a certain function on the space of lattices.
In Section \ref{sec.oscill} we show that for the resonant  terms
the numerators and denominators are asymptotically independent and finish the proof of Theorem \ref{main.limit}. In Section \ref{ScExt}, we show how the 
arguments of Sections \ref{ScNonRes}--\ref{sec.oscill} can be modified to prove some related results such as Theorem \ref{ThCont}(b). We also relate the discrepancy of Kronecker sequences 
to lattice counting problems (Section \ref{SSSlanted}) 
and study the visits of random geodesics to balls
(Section \ref{SSRandGeo}). 
The proof of Theorem \ref{ThCont}(a) which is simpler than the other proofs in the paper is given
in Section~\ref{Contd2}. In the last section of the paper, Section \ref{L},  we show that the series defining the limiting
distributions in Theorems \ref{main.limit} and~\ref{ThCont}
converge almost surely.
(A weaker statement that those series converge in probability follows from the proofs of Theorem \ref{main.limit}
and \ref{ThCont}. The convergence
in probability is sufficient for our argument. However we prove almost sure convergence since it provides an additional insight into the 
properties of the limiting distribution.)

\section{The limit distributions}
\label{s1}

\subsection{Notation.} \label{sn1}





Before we give a  formula for $\cD_\cC$ we introduce some notations related to the space of lattices that will be used in the statements and in the proofs. 

Let $M={\rm SL}(d+1,\R)/{\rm SL}(d+1,\Z)$. $M$ is canonically identified with
the space of unimodular lattices of $\R^{d+1}$.
Given $L \in M$ we denote by $e_1$ the shortest vector in $L.$ We then define inductively $e_2,\ldots,e_{d+1}$ such that for each $i \in [2,d+1]$, $e_i$ is the shortest vector in $L$ among those having the shortest nonzero projection on the orthocomplement of the plane generated by $e_1,\ldots,e_{i-1}$.
Clearly, the vectors $e_1(L),\ldots,e_{d+1}(L)$ are well defined outside a set of Haar  measure 0. Also, it is possible to show by induction on $d$ that the latter vectors generate the lattice (see \cite{Ar}, Lemma 49.3).

Let $\cZ$ be the set of primitive vectors $m \in \Z^{d+1}$ (i.e. with mutually coprime components) and such that if $i_0$ is the smallest integer in $[1,d+1]$ such that $m_i\neq 0$ then  $m_{i_0}>0$ (we add the latter condition to make sure not to count $-m$ in $\cZ$ for an  $m \in \cZ$).
Let 
\begin{equation}
T^\infty=\T^{d+1}\times \T^{\cZ}, \quad T_2^\infty=\T^{d+1}\times \T^{\cZ} \times \T^{\cZ}.
\end{equation}
We denote elements of $T^\infty$ by $(\theta, b)$ 
and the elements of $T_2^\infty$ by $(\theta, b, b').$
For $m \in \cZ$ and $L \in M$, we denote by $(m,e)$ 
the vector $\sum_{i \leq d+1} m_i e_i(L),$ by $X_m=(X_{m,1},\ldots,X_{m,d})$ 
its first $d$ coordinates, and by $Z_m$ its last coordinate. We also define $R_m={\left(\sum_{i\leq d }X_{m,i}^2\right)}^{\frac{1}{2}}$.




\subsection{Limit distribution in the case of translations.} 
Let $\mathcal{C}$ be a strictly convex body with smooth boundary. This means that $\partial \mathcal C$ is a smooth hypersurface of $\R^d$ with strictly positive  gaussian curvature, or equivalently that $\partial {\mathcal C}$ is  a smooth manifold isomorphic under the normal mapping to the unit sphere $\mathbb{S}_{d-1}$.  
For each vector $\xi \in {\mathbb S}_{d-1}$ there exists a unique point $x(\xi)\in \partial C$ at which the unit outer normal vector is $\xi$. We denote by $K(\xi)$ the gaussian curvature of $\partial C$ at this point.

Denote 
\begin{equation}
\cM_d=M\times T^\infty \text{ and }\cM_{2,d}=M\times T_2^\infty 
\end{equation}

By abusing sligtly the notation we let  $\mu$ denote the Haar measures on both $\cM_d$ and $\cM_{2,d}.$  
Consider the following function on $\cM_{2,d}$
\begin{equation}
\label{LatTor}
 \cL'_{{\mathcal C}}(L, \theta, b,b')=\frac{1}{\pi^2} \sum_{m\in \cZ}\sum_{p=1}^\infty k(p,m,\theta)
\frac{\sin (\pi p Z_m)}{R_m^{\frac{d+1}{2}} Z_m p^{\frac{d+3}{2}}}
\end{equation}
with 
\begin{multline} 
\label{DefK}
k(p,m,\theta)=  K^{-\frac{1}{2}}(X_m/R_m)  \sin(2\pi(p b_m+p(m,\theta) -(d-1)/8)) \\ +K^{-\frac{1}{2}}(-X_m/R_m)  \sin(2\pi(p b'_m-p(m,\theta)-(d-1)/8))\end{multline}
For the case of symmetric bodies, we define on the space $\cM_d$ the function  
\begin{multline}
\label{LimSumSym}
 \cL_{{\mathcal C}}(L, \theta, b)= \\  \frac{2}{\pi^2} \sum_{m\in \cZ}\sum_{p=1}^\infty
 K^{-\frac{1}{2}}(X_m/R_m) 
 \frac{\cos(2\pi p(m,\theta)) \sin(2\pi (p b_m-(d-1)/8)) \sin (\pi p Z_m)}
{R_m^{\frac{d+1}{2}} Z_m p^{\frac{d+3}{2}}}. 
\end{multline}

We now give the description of the distribution $\cD_{\mathcal C}$ of Theorem~\ref{main.limit} 
\begin{prop} 
\label{PrNS}
If ${\mathcal C}$ is an analytic non
symmetric strictly convex body in $\R^d$, then for any $z \in \R$ we have  
\begin{equation}
\label{distr.level}
\cD_{\mathcal C}(z) =  \mu\left\{ (L, (\theta, b,b'))\in \cM_{2,d}  : \cL'_{\mathcal C}(L,\th, b,b') \leq z \right\}.
\end{equation}
If ${\mathcal C}$ is symmetric then, for any $z \in \R$ we have 
\begin{equation}
\label{distr.level2}\cD_{\mathcal C}(z) =  \mu\left\{ (L, (\theta, b))\in \cM_d : \cL_{\mathcal C}(L,\th, b) \leq z \right\}.\end{equation}
\label{main3} \end{prop}

\begin{remark}
Note that $T^\infty$ is embedded into $T_2^\infty$ as a diagonal 
$$T^\infty=\{b_m'=b_m\}$$ 
and that $\cL_\cC'$ restricted to $T^\infty$ reduces to $\cL_\cC.$
Thus the proof of Theorem \ref{main.limit} will consist of two parts. First, we will see that for any analytic body the limiting distribution
will be given by \eqref{distr.level} where $\mu$ is a product of the Haar measure on $M$ and a Haar measure on a subtorus of
$T_2^\infty$ and, second,  we will show in sections \ref{secsec51} and \ref{secsec52} 
that the only subtori which can appear are $T^\infty$ and $T_2^\infty.$
\end{remark}

\begin{remark} {\rm We will see that the conclusions of Theorem \ref{main.limit} and of Proposition  \ref{main3}
 actually  hold 
for {\it generic} strictly convex symmetric bodies and {\it generic} strictly convex bodies respectively 
with a $C^{\nu}$ boundary where $\nu=(d-1)/2$. 
We will explain in  section \ref{sec.generic} 
what are the conditions required of these generic convex bodies.}
\end{remark}

\subsection{Limit distribution in the case of flows.} \label{sectflow}

In the  case of flows, we start by describing the limit distribution in the two-dimensional situation.
\begin{prop} \label{prop.theorem3a}  If ${\cC}$ is  analytic
strictly convex body in $\R^2$ that fits inside the unit cube, then the distribution of $\bD_\cC(v, x, T)$ of 
Theorem \ref{ThCont} $(a)$ converges as $T\to\infty$
to the distribution 
\begin{align} \label{dist.dim2} \brsD_{\cC, v} (z)&={\rm Leb} \left\{ (y,\theta) \in \T^2 \times \T^2 :  \cL_{v} (y,\theta) \leq z\right\} \\
 \cL_{v} (y,\theta) &=\sum_{k\in\Z^2-0} c_k e^{2\pi i (k, y)} \frac{\sin(\pi(k, \theta))}{\pi (k,v)}  \label{L.dim2} \end{align}
where $c_k$ are the Fourier coefficients of $\chi_\cC.$
\end{prop} 

The case $d=3$ is completely distinct and will be dealt with in \cite{DF2}. The limit distribution in the case $d\geq 4$ is as follows.  
\begin{prop} \label{prop.theorem3b} 
If $d \geq 4$ and ${\mathcal C}$ is  analytic
symmetric strictly convex body in $\R^d$ that fits inside the unit cube, then  the distribution $\sD_{\cC,v} (z)$ of Theorem \ref{ThCont} $(b)$ is given by  
\begin{equation} \label{dist.dimd1} \sD_{\cC,v} (z) =  \mu\left\{ (L, (\theta, b))\in \cM_d : \sL_v(L,\th, b) \leq z \right\} \end{equation} 
where 
\begin{multline} \label{dist.dimd2} \sL_{v}(L, \theta, b)= \\ \frac{2}{\pi^2} \sum_{m\in \cZ}\sum_{p=1}^\infty K^{-\frac{1}{2}}(X_m/R_m)  \frac{\cos(2\pi p(m,\theta)) \sin(2\pi (p b_m-(d-1)/8)) \sin (\pi p {\rho} Z_m)}  {  p^{\frac{d+3}{2}}  {\rho} Q_m^{\frac{d+1}{2}} Z_m}. \end{multline}
Here we write $v=\rho(\a_1,\ldots,\a_{d-1},1)$, $X_{m,s}$ and $R_m$ are defined as in Section \ref{sn1} with 
$(L, \theta, b) \in \cM_{d}$ instead of $(L, \theta, b) \in \cM_{d+1}$ and  
$$Q_m^2=R_m^2+\left(\sum_{s=1}^{d-1} {\alpha}_s X_{m,s} \right)^2.$$

In the case of non-symmetric  strictly convex body, the same statement holds 
except that the limiting distribution is given by
\begin{equation} \label{dist.dimd3} \sL'_{v}(L, \theta, b)= \frac{2}{\pi^2} \sum_{m\in \cZ}\sum_{p=1}^\infty   
k(p, m, \theta) \frac{ \sin (\pi p {\rho} Z_m)}  {  p^{\frac{d+3}{2}}  {\rho} Q_m^{\frac{d+1}{2}} Z_m}. \end{equation}
where $k(p, m, \theta)$ is given by \eqref{DefK}.
\end{prop}

\section{Non-resonant terms}
\label{ScNonRes}

In this section we study Fourier transform of the discrepancy  function and show that the main contribution
comes from a small number of resonant harmonics.

In all the sequel we fix $\eps>0$ arbitrarily small. We will use the notation $C$ for constants that may vary from one line to the other but that do not depend on anything but the dimension $d$.
 
\subsection{} \label{convex}  We shall use the asymptotic formula for the Fourier coefficients of the 
indicator function $\chi_\cC$ of a smooth strictly convex body $\cC$ 
obtained in \cite{herz}.

For any vector $t\in \R^d$ define $P(t)=\sup_{x \in \partial \cC} (t,x)$. The main result of \cite{herz} is that 
if $\cC$ is of class $C^\nu$ where $\nu=\frac{d-1}{2}$ then
\begin{equation}
(2\pi i |t|) \widehat{\chi}_{\mathcal C}(t)=\rho(\cC,t)-\bar{\rho}(\cC,-t)
\end{equation}
with 
\begin{equation}
\rho(\cC,t)={|t|}^{-\frac{d-1}{2}}K^{-\frac{1}{2}}(t/|t|) e^{i2\pi (P(t)-(d-1)/8)} +\cO(|t|^{-\frac{d+1}{2}}).
\end{equation}

 If we group the $k$ and $-k$ terms in the Fourier series we get 
\begin{align}
\label{fourier.strictly.convex}
\chi_{\cC_r}(x)-\Vol(\cC_r)&= {r^{\frac{d-1}{2}}} \sum_{k\in \Z^d-\{0\}} c_k(r,x) \\ \nonumber
c_k(r)&=d_k(r,x)
+\cO\left(|k|^{-\frac{d+3}{2}}\right)  \\ \nonumber 
d_k(r,x)&=    \frac{1}{2\pi}    \frac{g(k,r,x) +g(-k,r,x) }{{|k|}^{\frac{d+1}{2}}} \\Ê\nonumber
g(k,r,x)&= K^{-\frac{1}{2}}(k/|k|) \sin \left(2\pi ( rP(k)- (d-1)/8 +(k,x))\right)
\end{align}
which in the case of a symmetric body becomes 
\begin{align}
\label{fourier.strictly.convex.symmetric}
\chi_{\cC_r}(x)-\Vol(\cC_r)&= { r^{\frac{d-1}{2}}} \sum_{k\in \Z^d-\{0\}} c_k(r) \cos(2\pi(k,x))\\ \nonumber
c_k(r)&=d_k(r)
+\cO\left(|k|^{-\frac{d+3}{2}}\right)  \\ \nonumber 
d_k(r)&= \frac{1}{\pi}    \frac{g(k,r)  }{{|k|}^{\frac{d+1}{2}}}   \\
\nonumber
g(k,r)&=K^{-\frac{1}{2}}(k/|k|)\sin(2\pi(rP(k)-(d-1)/8)).
\end{align}

\subsection{} Throughout Section \ref{ScNonRes}, to simplify the notations in our manipulations of the Fourier series of the characterisitc functions of the sets included in $\cC_r$, we will assume the shape is symmetric and use therefore the formula (\ref{fourier.strictly.convex.symmetric}). We will see in Section \ref{sec.oscill} what are the necessary changes to be made in the case of a non symmetric body.

From now on we will use the notation, for $k=(k_1, \ldots,k_d)$ and $\a=(\a_1,\ldots,\a_d)$,  $\{ k, \a \} :=(k,\a)+{k_{d+1}}$ where ${k_{d+1}}$ is the unique integer such that  $-\frac{1}{2}<(k,\a)+{k_{d+1}}\leq \frac{1}{2}.$ 
To evaluate $D_{\mathcal C}(r,\a,x,N)  = \sum_{n=0}^{N-1} \chi_{{\mathcal C}_r}(T_\a^n x) - N \Vol(\cC_r)$, we sum up term by term in the Fourier expansion (\ref{fourier.strictly.convex.symmetric}) of $ \chi_{{\mathcal C}_r}$. Thus, introduce  the notation  
\begin{equation}
f(r,\a,x,N,k)=  
\end{equation}
$$ c_k(r) \frac{\cos(2\pi (k,x) +\pi(N-1)\{k,\a\}) \sin (\pi N\{k,\alpha\})}{ N^{\frac{d-1}{2d}} \sin(\pi \{k,\alpha\})} $$
so that we are interested in the distribution of 
\begin{equation}
\Delta(r,\a,x,N)  =  \sum_{k\in \Z^d-\{0\}} f(r,\a,x,N,k)
\end{equation}

\subsection{} Given a set $S$, for functions $h$ defined on $\T^{2d} \times S$, we denote by $\|h\|_2$ the supremum of the $L^2$ norms $\|h(\cdot,s)\|$ over all $s \in S$. Let
\begin{equation}\bar{\Delta}(r,\a,x,N) =  \sum_{k \in \Z^d-\{0\} \ : \ 0<|k|^2<\frac{N^{\frac{2}{d}}}{\eps}}  f(r,\a,x,N,k).\end{equation}
\begin{lemma} We have 
\begin{equation} \label{bsup} {{\| \Delta-\bar{\Delta}\|}_2} \leq C \eps^{1/4} \end{equation}
\end{lemma}
\begin{proof} We have that 
$$\int_{\T^d} {\left(\frac{\sin(\pi N (k, \a))}{\sin(\pi (k, \a))}\right)}^2 d\a \leq N. $$
Since $|d_r(k)|=\cO(|k|^{-\frac{d+1}{2}})$ we get that 
$${\| \Delta-\bar{\Delta}\|}_2^2 \leq  C N \frac{1}{{N^{\frac{d-1}{d}}}} \sum_{|k|^2\geq \frac{N^{\frac{2}{d}}}{\eps}}  \frac{1}{|k|^{d+1}} 
\leq C \sqrt{\eps}. \qedhere $$
\end{proof}

\subsection{} 

 Let  $$S(N,\a)=\left\{ k \in \Z^d-\{0\} : 0<|k|^2<\frac{N^{\frac{2}{d}}}{\eps}  ; |k|^{\frac{d+1}{2}} |\{k,\a\}| < \frac{1}{\eps^{\frac{d}{4}} N^{\frac{d-1}{2d}} } \right\}$$ 
\begin{equation}
\tD(r,\a,x,N) =  \sum_{k \in S(N,\a)} f(r,\a,x,N,k). 
\end{equation}

We have 
\begin{lemma}
\begin{equation} \label{bsup1} {{\| \Delta-\tD\|}_2} \leq C \eps^{1/8} \end{equation}
\end{lemma}
\begin{proof} By (\ref{bsup}) it is sufficient to show that  ${{\| \bar{\Delta}-\tD\|}_2^2} \leq C \eps^{1/4}$.  
We have 
\begin{equation}
 {\| \bar{\Delta}-\tD\|}_2^2 \leq \frac{C}{N^{\frac{d-1}{d}}} \sum_{|k|^2<\frac{N^{\frac{2}{d}}}{\eps}} A_{k}
\end{equation}
with 
\begin{equation}
A_{k} = \int_{\T^d}  \frac{c_k^2}{ {\{ k,\a}\}^2}   
\chi_{    |k|^{\frac{d+1}{2}}  |\{k,\a\}| \geq \frac{1}{ \eps^{\frac{d}{4}} N^{\frac{d-1}{2d}}   }} d\a
\end{equation}
For $p \geq 1$ we define 
\begin{equation}
B(k,p)= \left\{\a \in \T^d :  \frac{p}{ \eps^{\frac{d}{4}} N^{\frac{d-1}{2d}}   }\leq  |k|^{\frac{d+1}{2}}  |\{k,\a\}| \leq \frac{p+1}{ \eps^{\frac{d}{4}} N^{\frac{d-1}{2d}}   } \right\}.
\end{equation}
Then 
\begin{equation}
|B(k,p)|\leq \frac{1}{|k|^{(d+1)/2}\eps^{d/4} N^{\frac{d-1}{2d}}}.
\end{equation}
Thus 
\begin{equation}
A_k \leq    \sum_{p\geq 1} \frac{\eps^{d/4} c_k^2 |k|^{(d+1)/2}N^{\frac{d-1}{2d}}} {p^2 }
\leq C \eps^{d/4} N^{\frac{d-1}{2d}} c_k^2 |k|^{(d+1)/2}.
\end{equation}
Summing over $k$ and using \eqref{fourier.strictly.convex.symmetric} we get that 
\begin{equation}
\label{SumAK}
\sum_{|k|^2<\frac{N^{\frac{2}{d}}}{\eps}} A_{k} \leq C \eps^{1/4} N^{\frac{d-1}{d}}
\end{equation}
and the  claim follows.
\end{proof}

\subsection{} Let  $$\hat{S}(N,\a)=\left\{ k  \in \Z^d-\{0\}  : \eps^{\frac{d+4}{d-1}} N^{\frac{2}{d}} <|k|^2<\frac{N^{\frac{2}{d}}}{\eps}  ; |k|^{\frac{d+1}{2}} |\{k,\a\}| < \frac{1}{\eps^{\frac{d}{4}} N^{\frac{d-1}{2d}} } \right\}.$$ 
Define 
\begin{equation}
\hat{\Delta}(r,\a,x,N) =  \sum_{k \in \hat{S}(N,\a)}f(r,\a,x,N,k).
\end{equation}

Let
\begin{equation}
E_{k,N}= \left\{\a\in \T^d:  |k|^{\frac{d+1}{2}} |\{k,\a\}| 
< \frac{1}{\eps^{\frac{d}{4}} N^{\frac{d-1}{2d}} } \right\}
\end{equation}
and 
\begin{equation}
E_N=\bigcup_{|k|^2 <\eps^{\frac{d+4}{d-1}} N^{\frac{2}{d}}} E_{k,N}.
\end{equation}
We have that $|E_N|\leq C \eps$. On the other hand, since $\hat{\Delta}(r,\a,x,N) =\tD(r,\a,x,N)$ for $\a \notin E_{N}$ we have from (\ref{bsup1}) 
\begin{equation} \label{binf} {{\| \Delta-{\hat{\Delta}}\|}_{L^2((\T^d-E_N)\times \T^d)}} \leq C \eps^{1/8}. \end{equation}

\subsection{}ÊWe can now get rid of the error terms in the Fourier expansion of the characteristic functions of the convex sets. Introduce 
$$\check{f}(r,\a,x,N,k)= d_k(r) \frac{\cos(2\pi (k,x) +\pi(N-1)\{k,\a\}) \sin (\pi N\{k,\alpha\})}{N^{\frac{d-1}{2d}} \sin (\pi \{Êk,\alpha \}) }$$
and let 
\begin{equation}
\label{HDelta}
\check{\Delta}(r,\a,x,N) =  \sum_{k \in \hat{S}(N,\a)}\check{f}(r,\a,x,N,k).
\end{equation}
Since $|c_k-d_k|=\cO(|k|^{-(d+3)/2})$
\begin{equation} {\| \check{\Delta}-{\hat{\Delta}}\|}_{2}^2 \leq \sum_{\eps^{\frac{d+4}{d-1}} N^{\frac{2}{d}} <|k|^2<\frac{N^{\frac{2}{d}}}{\eps}} \frac{C}{|k|^{d+3}} \frac{N}{N^{\frac{d-1}{d}}} 
= \cO\left(N^{-\frac{2}{d}}\right). \end{equation}
Hence we can replace $\hat{\Delta}$ with $\check{\Delta}$.

\subsection{} \label{SS2}
Observe that the sum in \eqref{HDelta} is limited to large $k$ and small $|\{k,\a\}|$. Define 
$$g(r,\a,x,N,k)= d_k(r) \frac{\cos(2\pi (k,x) +\pi(N-1)\{k,\a\}) 
\sin (\pi N\{k,\alpha\})}{\pi N^{\frac{d-1}{2d}} \{k,\alpha\}}. $$

Thus we have to prove that
\begin{equation} \lim_{N \to \infty}   \lambda \{ (\a,x,r) \in \T^{2d} \times [a, b]
  \  | \  \Delta'(r,\a,x,N) \leq z\} = \cD(z)\end{equation}
where
\begin{equation}
\label{OnlyRes}
\Delta' =  \sum_{k\in U(N,\a)}   g(r,\a,x,N,k)
\end{equation}
and $U(N,\a)$ is any subset of $\Z^d$ that contains $\hat{S}(N,\a)$.

\section{Geometry of the space of lattices}
\label{ScGeom}

\subsection{} 
\label{sec.shortestvectors} Following \cite{Da}, Section 2,
we give now an interpretation of the set ${ \hat{S}(N,\a)}$, as well as the contribution to $\Delta'$ of each  $g(r,\a,x,N,k)$ for $ {k\in \hat{S}(N,\a)}$, in terms of short vectors in lattices in $M={\rm SL}(d+1,\R)/{\rm SL}(d+1,\Z)$.

Let 
$$ g_T=\left(\begin{array}{cccc} e^{-T/d} & 0 & \ldots & 0 \cr
                                                            0 & e^{-T/d} & 0 & \ldots  \cr
                                                            & & & \cr
                                                            0 & \ldots & e^{-T/d}  & 0 \cr 
                                                            0 & \ldots & 0 & e^T  \cr
                                                             \end{array}\right), \quad
\Lambda_\alpha=\left(\begin{array}{cccc} 1 & 0 & 0 & \ldots \cr
                                                            0 & 1 & 0 & \ldots  \cr
                                                            \cr
						   \a_1 & \ldots & \a_d & 1 \cr
 \end{array}\right) . $$
Consider the lattice $L(N,\a)=g_{\ln N} \Lambda_\alpha \Z^{d+1}.$ For each $k =(k_1, \ldots,k_d) \in \Z^d$ we associate the vector ${\bf k}=(k_1, \ldots,k_d, {k_{d+1}}) \in \Z^{d+1}$ 
where ${k_{d+1}}=k_{d+1}(k,\a)$ is the unique integer such that $-\frac{1}{2}<(k,\alpha)+{k_{d+1}}\leq \frac{1}{2}.$ 
We then denote 
\begin{equation}
(X_1,\ldots,X_d,Z) :=(k_1/N^{1/d},\ldots,k_d/N^{1/d},N\{k,\a\})=g_{\ln N} \Lambda_\alpha {\bf k}
\end{equation}

We have that $k \in { \hat{S}(N,\a)}$ if and only if  $g_{\ln N} \Lambda_\alpha {\bf k}$ satisfies 
\begin{equation} \label{cusp}  \eps^{\frac{d+4}{d-1}} <X_1^2+\ldots+X_d^2<\frac{1}{\eps}, \quad 
\left|Z\right|<\frac{1}{(X^2+\ldots+X_d^2)^{\frac{d+1}{4}} \eps^{\frac{d}{4}}}.  \end{equation}

Let $e_i(N,\a)$ be the shortest vectors of $L(N,\a)$ as defined in Section~\ref{s1}.

\begin{lemma} 
For each $\eps>0$ there exists $M(\eps)>0$ such that if $\alpha\not\in E_N$ then
$k\in \hat{S}(N,\a)$ implies that 
$$g_{\ln N} \Lambda_\alpha {\bf k}=m_1 e_1(N,\a)+\ldots+m_{d+1} e_{d+1}(N,\a)$$
for some unique $(m_1,\ldots,m_{d+1})\in \Z^{d+1}-(0,\ldots,0)$, $\|m\|\leq M(\eps)$.

If $\eps>0$ is fixed and $N$ is sufficiently large, it also holds that if $\alpha\not\in E_N$ then
for each $\|m\|\leq M(\eps)$, 
there exists a unique $k \in \Z^d$ 
such that 
$$g_{\ln N} \Lambda_\alpha {\bf k}=(m,e(N,\a))=m_1 e_1(N,\a)+\ldots+m_{d+1} e_{d+1}(N,\a). $$
\end{lemma}

We denote $U(N,\a,\eps)$ the set of $k\in \Z^{d}$ that correspond to the set of $m \in \Z^{d+1}, \|m\|\leq M(\eps)$. 

\begin{proof} It is clear from (\ref{cusp}) that 
$k \in \hat{S}(N,\alpha)$ implies that $g_{\ln N} \Lambda_\alpha {\bf k}$ is shorter than $R(\eps)=\eps^{-\frac{(d+4)(d+1)}{4(d-1)}-1}.$ Since $e_1(L)\dots e_{d+1}(L)$ is a basis in $\R^{d+1}$ we have that the norms 
$||x||$ and $||\sum_j x_j e_j(L)||$ are equivalent.
Accordingly for each $L$ there exists $M(L)$ such that
$\|m_1 e_1(L)+\ldots+m_{d+1} e_{d+1}(L)\|\geq R(\eps)$ provided that $||m||\geq M(L).$ We claim that $M(L)$ can
be chosen uniformly for $L$ of the form $L(N, \a)$ with $\a\not\in E_N.$ To this end it suffices to show that the
set 
\begin{equation}
\label{LnotinEN}
\{L(N,\a), \a\not\in E_N\}
\end{equation}
is precompact.
By definition of $E_N$, if   $X_1^2+\ldots+X_d^2<\eps^{\frac{d+4}{d-1}} $, then $N|\{k,\a\}|$ is large, hence $|N((k,\a)+k_{d+1})|$ is {\it a fortiori} large for any $k_{d+1} \in \Z^d$. This implies that there exists $\delta(\eps)$ such that 
if $\alpha\not\in E_N$ then all vectors in $L$ are longer than $\delta.$ 
Therefore the precompactness of \eqref{LnotinEN} follows by
Mahler compactness criterion (\cite{R}, Corollary 10.9). 

We now prove the  second statement. We have that $(m,e(N,\a))=g_{\ln N} \Lambda_\alpha \bar{k}$ for some unique $\bar{k} \in \Z^{d+1}$ and we just have to see that $\bar{k}={\bf k}(k)$ for 
$k=(\bar{k}_1,\ldots,\bar{k}_d)$. Since for $\|m\|\leq M(\eps)$ we have that $\|(m,e)\| \ll N$ 
(by precompacity) we necessarily have $\bar{k}_{d+1}={k_{d+1}}(k,\a)$, that is $\bar{k}={\bf k}(k)$ as required. 
\end{proof}

\subsection{} 
\label{SSContr}

For $m \in \Z^{d+1}$ and $\a \in \T^d$,  we write 
\begin{equation}
\label{EqDefX}
(m,e(N,\a))=(X_{m,1},\ldots,X_{m,d},Z_m) 
\end{equation}
and define $X_m=(X_{m,1},\ldots,X_{m,d})$ and $R_m=\|X_m\|.$
Introduce
\begin{equation*} h(r,\a,x,N,m)= \frac{d_r(N,m)  \cos(2\pi N^{1/d} (X_m,x) +\frac{\pi (N-1)}{N}Z_m) \sin (\pi Z_m)}{ {R_m^{\frac{d+1}{2}}}  Z_m} \end{equation*}
with 
$$ d_r(N,m)= \frac{1}{\pi^2}  K^{-\frac{1}{2}}(X_m/R_m)\sin(2\pi(rN^{1/d}P(X_m)-(d-1)/8))$$

From Section \ref{sec.shortestvectors} we see that for $\a\not\in E_N$ 
\begin{equation}
\sum_{m\in \Z^{d+1}-\{0\}, \|m\|\leq M(\eps)}   h(r,\a,x,N,m) =  \sum_{k\in U(N,\a,\eps)}   g(r,\a,x,N,k)
\end{equation}
where $U(N, \alpha, \eps)\supset S(N, \alpha).$

Therefore Section \ref{SS2} allows to shift our attention to the distribution of
$\sum_{m\in \Z^{d+1}-\{0\}, \|m\|\leq M(\eps)}   h(r,\a,x,N,m)$ that is equivalent to the distribution of $\Delta'$ that we are studying. 

The idea now is that the variables  $r N^{1/d} P(X_m)  \ {\rm mod \ }  [1]$, as $r$ is random in an interval,  will behave as uniformly distributed random variables on the circle, provided that only prime vecotrs $m$ are considered. 
We need however to account for the contribution of the multiples of the primitive vectors. Introduce 
\begin{multline}  \label{qq} q(r,\a,x,N,m,p)= \\  \frac{d_r(N,m,p)  \cos\left(2\pi p (m,{\gamma}(\a,x,N)\right) +p\frac{\pi (N-1)}{N}Z_m) \sin (\pi  p Z_m)}{ {R_m^{\frac{d+1}{2}}} Z_m p^{\frac{d+3}{2}}} \end{multline} 
where
\begin{equation}  
d_r(N,m,p)=\frac{1}{\pi^2} K^{-\frac{1}{2}}(X_m/R_m)\sin(2\pi(rN^{1/d}pP(X_m)-(d-1)/8)), 
\end{equation}
\begin{equation}  
\label{EqDefGamma}
{\gamma}(\a,x,N)=(\gamma_1(\a,x,N),\ldots,\gamma_{d+1}(\a,x,N)),
\end{equation} 
\begin{equation}  
\gamma_j(\a,N,x)=N^{1/d} (e_{j,1}(N,\a)x_1+\ldots+e_{j,d}(N,\a)x_d). 
\end{equation}

Recall the definition of $\cZ$ in Section \ref{s1}.
Let $\cZ_\eps= \{m\in \cZ \ : \ \|m\|\leq M(\eps)\}.$  
Summing over the multiples of all $m\in \cZ_\eps$ 
we end up with the following  
\begin{prop} \label{finalreduction} If as $\a,x,r$ are  uniformly distributed 
on $\T^d \times \T^d \times [a,b]$, the variable 
$$2 \sum_{p=1}^\infty \sum_{m\in \cZ_\eps}  q(r,\a,x,N,m,p)$$ 
converges in distribution as $N\to \infty$ and then $\eps \to 0$ 
to some law $\cD_{\mathcal C}(z)$ 
then  the limit 
(\ref{limit2}) of Theorem \ref{main.limit} holds with the same limit law  $\cD_{\mathcal C}(z)$. 
\end{prop}

To proceed with the proof of  Theorem \ref{main.limit} we thus need to see how the terms  $X_m, R_m, Z_m,\gamma(\a,x,N)$ and $r  N^{1/d} P(X_m)$   behave as $\a,x,r$ are random and $N\to \infty$.

\subsection{Uniform distribution of long pieces of horocycles} 
\label{SSEq}
Observe that $\Lambda_\a$ is a piece of unstable manifold of $g_T.$ We shall use the fact
that the images of unstable leaves became uniformly distributed in $M.$ 
The statement below is a special case of \cite{MS1}, Theorem 5.8.
Related results are proven in several papers, see, in particular \cite{EM, KM, Sh}.

\begin{prop} Denote by $\mu$  the Haar measure on $M.$ If $\Phi:(\R^{d+1})^{d+1} \times \R^d \to \R$ is a bounded continuous function then
\begin{multline}
\label{EquidWu2}
\lim_{N \to \infty} \int_{\T^d} \Phi\left(e_1(L(N,\a)), \dots, e_{d+1}(L(N,\a)),\a\right) d\a = \\ \int_{M \times \T^d} \Phi(e_1(L), \dots, e_{d+1}(L),\a) d\mu(L) d\a \end{multline}

\end{prop}


\section{Oscillating terms} \label{sec.oscill}

Recall the definitions of $\gamma$ and $X_m$ given in section \ref{SSContr} (equations \eqref{EqDefX}
and \eqref{EqDefGamma}). 
Recall also the definition of the function $P(t)=\sup_{x \in \partial \cC} (t,x)$.  We denote by $\mu_d$ the distribution of $e_1(L),\ldots, e_{d+1}(L)$ when $L$ is distributed according to Haar measure on $M={\rm SL}(d+1,\R)/{\rm SL}(d+1,\Z)$. We denote by $\lambda_{d,\eps}$ the Haar measure on $\T^{d+1} \times \T^{\cZ_\eps}$
and by $\bar \lambda_{d,\eps}$ the Haar measure on $\T^{d+1} \times \T^{\cZ_\eps} \times  \T^{\cZ_\eps}$.

The goal of this section is to prove the following. 
\begin{prop} \label{prop.independence} If $\a,x,r$ are distributed with smooth densities on $\T^d \times \T^d \times [a,b]$,  the random variables 
\begin{equation} e_1(N,\a),\ldots, e_{d+1}(N,\a), \quad 
\{\gamma_j\}_{j=1}^{d+1}, \quad 
\{A_m  \}_{m\in \cZ_\eps}
\end{equation}
with  $A_m=N^{\frac{1}{d}} P(X_m)r$,  converge in distribution as $N \to \infty$ to $\mu_d \times \lambda_{d,\eps}$.
In the non symmetric case,  the distribution of the random variables 
 \begin{equation}
e_1(N,\a),\ldots, e_{d+1}(N,\a), \quad  \{\gamma_j\}_{j=1}^{d+1}, \quad \{A_m  \}_{m\in \cZ_\eps}, \quad \{ \bar{A}_m \}_{m\in \cZ_\eps} \end{equation}
where $\bar{A}_m=N^{\frac{1}{d}} P(-X_m) r$, converge in distribution as  $N \to \infty$  to  $\mu_d \times \bar{\lambda}_{d,\eps}$.
\end{prop} 

We will prove Proposition \ref{prop.independence} in Section \ref{SSParts}. We will first prove in Section  
\ref{secsec52} 
that for $m_1,\ldots,m_K \in \cZ$, the   $\{P(X_{m_i})\}_{i=1}^K$ are typically 
independent over $\Q$ and in the non symmetric case we want to prove that for $m_1,\ldots,m_K \in \cZ$, 
the   $\{P(X_{m_i})\}_{i=1}^K$ and $\{P(-X_{m_i})\}_{i=1}^K$ are typically independent over $\Q$. 
The precise statements to which this section is devoted are enclosed in equations   (\ref{AlmRes}) and  (\ref{AlmRes2}) at the end of Section \ref{secsec52}. We will first need two auxiliary lemmas about the function $P$ that we include in the next section.

\subsection{}
\label{secsec51}
For any $L \in \text{GL}(d,\R)$ viewed as a linear invertible map of $\R^d$, we define $f_L:\R \to \R $ as $f_L(\delta)=(P \circ L) (1,\delta,0,\ldots,0)$. We also denote $\tilde{f}_L(\delta)=(P \circ L) (-1,-\delta,0,\ldots,0)$.

\begin{lemma} \label{analytic} If $\cC$ is real analytic we have that for any $L \in \text{GL}(d,\R)$, $f_L$ is real analytic and not equal to a polynomial. 
\end{lemma}

\begin{proof} We have that $f_L(\delta)=\sqrt{1+\delta^2} (P \circ L) \left(\frac{1}{\sqrt{1+\delta^2}}, \frac{\delta}{\sqrt{1+\delta^2}},0,\ldots,0\right)$. 
Suppose $f_L$ is a polynomial. Observe that 
$ (P \circ L) \left(\frac{1}{\sqrt{1+\delta^2}}, \frac{\delta}{\sqrt{1+\delta^2}},0,\ldots,0\right)$ 
is bounded so that $f_L$ can only be of degree at most  $1$. Since $f_L$ is strictly positive and not constant this leads to a contradiction. 
\end{proof}

We will need the following lemma for the non symmetric case.
\begin{lemma}
\label{LmNonSym} The following alternative holds. Either
\begin{itemize}
\item[(i)] There exists $L \in \text{GL}(d,\R)$ and $\delta, \delta' \in \R$ such that 
\begin{equation}
\label{WSym}
\frac{f^{(2)}_L(\delta)}{f^{(2)}_L(\delta')}\neq \frac{\tf^{(2)}_L(\delta)}{\tf^{(2)}_L(\delta')} 
\end{equation}
or 
\item[(ii)] $\cC$ has a center of symmetry.
\end{itemize}
\end{lemma}


\begin{proof}
Suppose that (i) does not hold. Let $L= Id$. We have that  $f^{(2)}=c \tf^{(2)}$ for some constant $c$. 
In other words
\begin{equation}
\left(\frac{\partial}{\partial \delta}\right)^2 P(1, \delta, 0, \ldots, 0)=
c \left(\frac{\partial}{\partial \delta}\right)^2 P(-1, -\delta, 0, \ldots, 0). 
\end{equation}
Since for $x>0$ we have
$P(x,y, 0, \ldots, 0)=xP(1, y/x, 0, \ldots, 0)$ it follows that
\begin{equation}
 \partial_y^2 P(x, y, 0, \ldots, 0)=c \partial_y^2 P(-x, -y, 0, \ldots, 0)
\end{equation}
for $x>0.$ Since $\cC$ is analytic, this equality in fact holds identically.
In particular
\begin{equation}
\partial_y^2 P(-x, -y, 0, \ldots, 0)=c \partial_y^2 P(x, y, 0, \ldots, 0)
\end{equation}
so that $c=\pm 1.$
Rewriting the last equation as
\begin{equation}
\partial_y^2 \left[P(-x, -y, 0, \ldots, 0)-c P(x, y, 0, \ldots, 0) \right]=0 
\end{equation}
we conclude that
\begin{equation}
P(x,y,0,\ldots,0)-cP(-x,-y,0,\ldots,0)=a(x)+b(x)y
\end{equation}
Assuming that $0\in \cC$ we have that
both $P(x,y,0,\ldots,0)$ and $P(-x,-y,0,\ldots,0)$ are positive. This implies that $c=1.$ 
Indeed substituting $x=y=0$ we see that $a(0)=0$ and if $c$ were equal to $-1$ we would get 
\begin{equation}
P(0,y,0,\ldots,0)+P(0,-y,0,\ldots,0)=b(0)y.
\end{equation}
Since the RHS can not be positive for all $y$ we get a contradiction proving that $c$ is actually equal to $1.$

Interchanging the roles of $x$ and $y$, we get 
$$P(x,y,0,\ldots,0)-P(-x,-y,0,\ldots,0)=ax+by$$
Because the same reasoning holds for any choice of $L    \in \text{GL}(d,\R)$ we get that the
 restriction of the function
$P(x)-P(-x)$ to every plane is linear. Therefore this function is 
globally linear,
that is, there exists $v \in \R^d$ such that for every $x\in \R^d$
$$P(x)-P(-x)=(x,v). $$

Note that shifting the origin to $x_0$ replaces 
$P(x)$ by $P(x)+(x,x_0)$ and $P(-x)$ by $P(-x)-(x,x_0).$ 
Therefore after shifting the origin to $v/2$ we get
$P(x)=P(-x)$ so that $\cC$ is symmetric.
\end{proof}

\subsection{} \label{secsec52}

When $\cC$ is not symmetric, we will assume WLOG that (\ref{WSym}) holds for $L=Id$.     
For $m \in \Z^{d+1}$ define the function $p_{m}:\R^{2(d+1)}\to \R : (x,y)\mapsto P((m,x),(m,y),0,\ldots,0)$. We know from Lemma \ref{analytic} that  the function $f(\delta)=P(1,\delta,0,\ldots,0)$ is not a polynomial. 
In case the body $\cC$ is not symmetric we also consider 
$$\tilde{f}(\delta)= P(-1,-\delta,0,\ldots,0)  \text{ and }\tp_m=P(-(m,x), -(m,y), 0, \ldots, 0).$$  


\begin{prop} \label{independence} For any $m_1,\ldots,m_K \in \cZ$, if $l_1, \ldots, l_K $ are such that 
$\sum_{i=1}^K l_i p_{m_i} \equiv 0$, then $l_i=0$ for $i=1,\ldots,K$. 

If $\cC$ is non symmetric 
we have that  for any $m_1,\ldots,m_K \in \cZ$, if $l_1, \ldots, l_K $,  $\tilde{l}_1, \ldots, \tilde{l}_K, $  are such that 
$\sum_{i=1}^K l_i p_{m_i}+\sum_{i=1}^K \tilde{l}_i \tilde{p}_{m_i} \equiv 0$ then $l_i=\tilde{l}_i=0$ for $i=1,\ldots,K$.
\end{prop}

\begin{proof} Assume that  $\sum_{i=1}^K l_i p_{m_i} \equiv 0$. We fix $j$ and show that $l_j=0$. Fix $\beta \in \R^{d+1}$ such that $(m_j,\beta)\neq 0$. For $\a \in \R^{d+1}$ and $\delta, \theta \in \R$ we let $x=\a, y=\delta \a + \theta \beta$. Then $p_m(x,y)=|(m,\a)|f(\delta+\theta \frac{(m,\beta)}{|(m,\a)|})$ if $(m,\a)>0$ and 
$p_m(x,y)=|(m,\a)|\tilde{f}(\delta+\theta \frac{(m,\beta)}{|(m,\a)|})$ if $(m,\a)<0$.

Fix $\a, \delta$ and expand the sum in powers of $\theta.$ Equating to
zero the term in front of $\th^2$ we get 
\begin{equation} \label{thet2} 
\sum_{i=1}^K   h_i \frac{(m_i,\beta)^2}{|(m_i,\a)|}=0 
\end{equation}
where $h_i=l_i f''(\delta)$ if $(m_i,\a)>0$ and $h_i=l_i \tilde{f}''(\delta)$ if $(m_i,\a)<0$. 
Now since   $m_1,\ldots,m_K$ are primitive vectors it is possible to choose $\a$ so that $(m_j,\a)>0$ is arbitrary small while $|(m_i,\a)|$ remain bounded away from zero for every $i \neq j$. Thus, we must have that $h_j=0$ and since there exists $\delta$ such that $f''(\delta)\neq 0$ (because $f$ is not a polynomial) we get $l_j=0$. 

When 
 $\sum_{i=1}^K l_i p_{m_i}+\sum_{i=1}^K \tilde{l}_i \tilde{p}_{m_i} \equiv 0$,  \eqref{thet2} becomes
\begin{equation} \label{thet3} 
\sum_{i=1}^K  h_i \frac{(m_i,\beta)^2}{|(m_i,\a)|}=0 
\end{equation}
where $h_i=l_i f''(\delta)+\tilde{l}_i \tilde{f}''(\delta)$ if $(m_i,\a)>0$ and $h_i=l_i \tilde{f}''(\delta)+\tilde{l}_i {f}''(\delta)$ if $(m_i,\a)<0$.  
Consider, for example, the case where the first alternative holds. 
As before, we must have $l_j f''(\delta) +\tilde{l}_j \tilde{f}''(\delta)=h_j=0$ for any choice of $\delta$. Since we assumed  
  (\ref{WSym}) holds for $L=Id$, this yields $l_j=\tilde{l}_j=0$.
\end{proof}

As a consequence of Proposition \ref{independence} we have the following facts. For any  $l_1,\ldots,l_K$ and any $m_1,\ldots,m_K \in \cZ$
\begin{equation}
\label{ResLat}
 \text{Leb} \left(z_1,\ldots,z_d \in (\R^{(d+1)})^d :  \sum_{i=1}^k l_i P((m_i,z_1),\ldots,(m_i,z_d))  =0\right)=0.
\end{equation}
 
So, if we take a lattice $L$ and denote 
$z_j=(e_{j,1}(L),\ldots,e_{j}(L))$, then $P(X_m(L))=P((m,z_1),\ldots,(m,z_d))$, and for any $l_1,\ldots,l_K$ and any $m_1,\ldots,m_K \in \cZ$
\begin{equation}
\label{ResLat}
 \mu\left(L:  \sum_{i=1}^k l_i P(X_{m_i}(L))=0\right)=0
\end{equation}
Now \eqref{EquidWu2} implies that 
\begin{equation}
\label{AlmRes}
\mes\left(\a\in \T^d: \left|\sum_{i=1}^K l_i P(X_{m_i}(L(N,\a)))\right|<\eps\right)\to 0 \text{ as } \eps \to 0, N\to\infty.
\end{equation}

Similarly, in the non symmetric case, it holds that for any $l_1,\ldots,l_K$, $\tilde{l}_1,\ldots,\tilde{l}_K$,  and any $m_1,\ldots,m_K \in \cZ$ and
\begin{equation} \label{AlmRes2} 
\mes\left(\a\in \T^d: \left|\sum_{i=1}^K l_i P(X_m)+\sum_{i=1}^K \tilde{l}_i P(-X_m)\right|<\eps\right)\to 0 
\end{equation}
 as  $\eps \to 0, N\to\infty.$
\subsection{Proof of Proposition \ref{prop.independence}.}
\label{SSParts}

We consider the case when $\cC$ is symmetric. The case when it is non
symmetric is similar.
Take integers $n_1,\ldots, n_{d+1},$  
$\{l_m\}_{m\in \cZ_\eps}$ and a function $\Phi:(\R^{d+1})^{d+1}\to \R$ of compact support.
We need to show that as $N\to\infty$
\begin{equation}
\label{AsInd} 
 \iiint \Phi(e_1(N, \a),\dots, e_{d+1}(N, \a))
 \exp\left[ 2\pi i \left(\sum_{j=1}^{d+1} n_j \gamma_j +\sum_{\cZ_\eps} l_m A_m\right)\right] dx d\a dr \to
\end{equation}
$$ \int_{M} \Phi(e_1(L),\dots e_{d+1}(L)) d\mu(L) 
\int_{\T^{d+1}} e^{2\pi i \sum_j n_j \gamma_j} d\gamma
\int_{\T^{\Z_\eps}} e^{2\pi i\sum_m l_m A_m} dA,$$
  as  $N\to \infty.$
 In case $n_j\equiv 0$ and $l_m\equiv 0$ the result follows from \eqref{EquidWu2}.
 
Therefore we may assume that some $n_j$ or some $l_m$ are non-zero 
so that \eqref{AsInd} reduces to 
 \begin{equation}
\label{AsInd2A} 
 \iiint \Phi(e_1(N, \a),\dots, e_{d+1}(N, \a))
 \exp\left[ 2\pi i \left(\sum_{j=1}^{d+1} n_j \gamma_j +\sum_{\cZ_\eps} l_m A_m\right)\right] dx d\a dr \to 0.
\end{equation}

 Suppose first that   $n_j \neq 0$ for at least one $j$. Recall the definition $\gamma_j(\a,N,x)=N^{1/d} (e_{j,1}(N,\a)x_1+\ldots+e_{j,d}(N,\a)x_d).$  Hence the coefficient in front of $x_1$ in $\sum_j n_j \gamma_j$ equals to
  $N^{1/d} \sum_j n_j e_{j,1}$. 

 Note that for almost every $L$
  the numbers $e_{1,1}(L),\ldots,e_{d,1}(L)$ are independent over $\Z$.  Hence \eqref{EquidWu2} implies that 
\begin{equation}
\label{LargeMostAlpha}
\mes\left(\alpha\in \T^d: 
\left|\sum_j n_j e_{j,1}(N,\a)\right|<\frac{1}{N^{\frac{1}{2d}}}\right) \to 0 
\end{equation}
as $N \to \infty$.
We thus split the LHS of \eqref{AsInd2A} into two parts where $I$ includes the integration over $\alpha$ with
$|\sum_j n_j e_{j,1}|<N^{-\frac{1}{2d}}$ and $\RmII$ includes the integration over $\alpha$ with 
$|\sum_j n_j e_{j,1}|\geq N^{-\frac{1}{2d}}.$ Then
$$|I|\leq \Const(\Phi) \mes(\alpha\in \T^d: |\sum_j n_j e_{j,1}|<N^{-\frac{1}{2d}}) $$
so it can be made as small as we wish in view of \eqref{LargeMostAlpha}.
On the other hand in $\RmII$ we can integrate by parts with respect to
$x_1$ and obtain the estimate
$$ |\RmII|\leq \frac{\Const(\Phi)}{ N^{\frac{1}{2d}}}. $$
This concludes the proof in case not all $n_j$ vanish. 

Similarly if not all $l_m$ vanish then we can integrate with respect to $r$ instead of $x_1$ 
using \eqref{AlmRes} instead of \eqref{LargeMostAlpha} to
obtain \eqref{AsInd2} in that case.
\hfill $\Box$

\subsection{Proof of Theorem \ref{main.limit} }
Combining Proposition \ref{finalreduction} and Proposition \ref{prop.independence} we obtain  Theorem \ref{main.limit} and Proposition \ref{main3} by letting $\eps\to 0.$ 

\hfill $\Box$


\subsection{Generic convex bodies} \label{sec.generic}  Observe that the fact that $\cC$ is real analytic is used only in Section \ref{sec.oscill} to prove \eqref{ResLat}.
For the rest of the argument it is enough that $\cC$ is of class $C^\nu$ where $\nu=\frac{d-1}{2}$ so that we can 
apply the results of \cite{herz} to get the asymptotics of the Fourier coefficients of $\chi_\cC.$

\begin{define}
We say that a convex body $\cC$ is {\it generic} if for any $K \in \N^*$, and any nonzero vectors $\ell=(l_1,\dots l_K, \tl_1, \ldots, \tl_K) \in \Z^{2K}$ and  
$M=(m_1,\dots m_K) \in \Z^K$ and any $\eta>0$, there exists $\eps>0$ such that 
\begin{equation}
\label{Res}
 \mu\left(L: \left|\sum_{i=1}^K \left[l_i P(X_{m_i}(L)+\tl_i P(-X_{m_i}(L))\right]\right|<\eps\right)<\eta. 
\end{equation}
\end{define}

Let $\cB(\eps,\eta,\ell,M)$ be the set of bodies of class $C^\nu$ such that (\ref{Res}) holds. This is clearly an open set and $\bigcup_{n \in \N^*} \cB(1/n,\eta,\ell,M)$ is dense since it contains  real-analytic non symmetric convex bodies. Therefore the set of {\it generic} bodies $\bigcap_{K\in \N^*} \cap_{(\ell,M)\in \Z^{3K}} \bigcap_{jÊ\in \Z^*} \bigcup_{n \in \N^*} \cB(1/n,1/j,\ell,M)$ is generic in the $C^\nu$ topology.

By the foregoing discussion we have
\begin{coro}
Theorem \ref{main.limit} is valid for generic convex bodies of class $C^r$ with $r\geq \nu$, and the limit distribution is given by Proposition \ref{main3}.
\end{coro}

\begin{remark} {\rm One defines in a similar way a class of {\it generic} symmetric bodies within the symmetric convex bodies of class $C^\nu$ where $\nu=\frac{d-1}{2}$ for which Theorem \ref{main.limit} will hold with a limit distribution  given by Proposition \ref{main3}.
}

\end{remark}

\section{Extensions}
\label{ScExt}
\subsection{Small balls.} \label{sec61} The analysis given above also applies to small copies of a given convex set.

\begin{theo} \label{short} 
Take $\gamma<1/d.$ For any $\mathcal{C}$ striclty convex analytic body 
for any $b>a>0$,  we have 
\begin{equation*} 
\lim_{N \to \infty} \frac{1}{b-a} \lambda \{ (r,\a,x) \in [a,b] \times \T^d\times \T^d 
\Big|  \frac{D_\cC(r N^{-\gamma} ,\a,x,N)}{r^{\frac{d-1}{2}} N^{\frac{d-1}{2d}(1-\gamma d) } } \leq z\}  =\cD_{\mathcal C}(z)
\end{equation*}
where $\cD_\cC(z)$ is the same as in Theorem \ref{main.limit}.
\end{theo} 
\begin{proof} The proof is very similar to the proof of Theorem \ref{main.limit} so we only describe the necessary
modifications.
We consider the case of symmetric bodies, the non-symmetric case requires straightforward modifications. We have
\begin{equation}
\label{DescSmallB}
\frac{D_\cC(r N^{-\gamma} ,\a,x,N)}{r^{\frac{d-1}{2}} N^{\frac{d-1}{2d}(1-\gamma d) } }=
\end{equation}
$$\sum_{k\in \Z^d-\{0\}} c_k^*(r) \frac{\cos(2\pi (k,x) +\pi(N-1)\{k,\a\}) \sin (\pi N\{k,\alpha\})}{ N^{\frac{d-1}{2d}} \sin(\pi \{k,\alpha\})}
$$
where $c_k^*(r)=N^{\frac{(d-1)\gamma}{2}}c_k(r N^{-\gamma}).$
Making the change of variables which rescales  
$\cC_{rN^{-\gamma}}$ to a unit size we get
see that for $|k|\geq N^{\gamma}$ we have
\begin{equation}
\label{LargeK}
 c_k^*=d_k^*(r)\left(1+\cO\left(\frac{N^\gamma}{|k|}\right)\right)
\end{equation}
where
\begin{equation}
d_k^*(r)= \frac{1}{\pi}    \frac{g(k,r N^{-\gamma})  }{{|k|}^{\frac{d+1}{2}}}.
\end{equation}
On the other hand if $|k|<N^{-\gamma}$ we have an {\it a priori} bound 
\begin{equation}
\label{SmallK}
c_k^*=\cO\left(N^{\frac{(d-1)\gamma}{2}} \Vol(C_{r N^{-\gamma}})\right)=
\cO\left(N^{\frac{-(d+1)\gamma}{2}}\right). 
\end{equation}
Now repeating the computations of Section \ref{ScNonRes} we obtain that
$\frac{D_\cC(r N^{-\gamma} ,\a,x,N)}{r^{\frac{d-1}{2}} N^{\frac{d-1}{2d}(1-\gamma d) } }$
is well approximated by 
\begin{equation}
\label{OnlyResSB}
\sum_{k\in U(N,\a)} d_k^*(r) \frac{\cos(2\pi (k,x) +\pi(N-1)\{k,\a\}) 
\sin (\pi N\{k,\alpha\})}{\pi N^{\frac{d-1}{2d}} \{k,\alpha\}}. 
\end{equation}
Note that in Section \ref{ScNonRes} we only use the bound on the absolute value of the Fourier coefficients.
So the only place where the argument has to be modified is the derivation of \eqref{SumAK}. Namely instead of
using \eqref{fourier.strictly.convex.symmetric} for all $k$ we have to use \eqref{SmallK} for $|k|<N^{-\gamma}$ and 
\eqref{LargeK} for $|k|\geq N^{-\gamma}.$ However the main contribution comes from
the terms where $|k|\geq N^{-\gamma}$ ensuring the validity of \eqref{SumAK}.

\eqref{OnlyResSB} is the same as \eqref{OnlyRes}
except that $r P(k)$ is replaced by $r N^{-\gamma} P(k).$
The explicit form of this term was only used in the proof of Proposition \ref{prop.independence}
 where we have used
that $r |k|\gg 1$ (namely, in sections  \ref{sec.shortestvectors} and 
\ref{SSContr}
 we had $|k|$ of the order of $N^{1/d}$ and we wrote $r P(k)=rN^{1/d} P(k/N^{1/d})$ and we used $rN^{1/d} \to \infty$). In the present setting $r P(k)$ is replaced by $r N^{-\gamma} P(k)$ and we still have $r|k| N^{-\gamma} \to \infty$ since the main
contribution for the discrepancy comes from $|k|\sim N^{1/d}.$ Hence the proof proceeds as before.
\end{proof}

\begin{remark}
While the limiting distributions for 
$\frac{D_\cC(r N^{-\gamma} ,\a,x,N)}{r^{\frac{d-1}{2}} 
N^{\frac{d-1}{2d}(1-\gamma d) } }$
are the same for all $\gamma$ if we fix $\a$ and $r$ then for $\gamma_1\neq \gamma_2$
$$\frac{D_\cC(r N^{-\gamma_1} ,\a,x,N)}{r^{\frac{d-1}{2}} 
N^{\frac{d-1}{2d}(1-\gamma_1 d) } }\not \approx 
\frac{D_\cC(r N^{-\gamma_2} ,\a,x,N)}{r^{\frac{d-1}{2}} 
N^{\frac{d-1}{2d}(1-\gamma_2 d) } }. $$ 
Namely while the small denominators will be the same in both cases the 
terms $\sin \left(2\pi ( r N^{-\gamma} P(k)- (d-1)/8 +(k,x))\right)$ in the numerators 
will be asymptotically independent for different $\gamma$s.
\end{remark}

\subsection{Parametric families of convex sets.}
We shall need the following extension of Theorem 
\ref{main.limit}. 
Assume that we have
an analytic family of convex sets $\{\cC_\alpha\}_{\alpha\in \T^d}.$ That is, we assume that
$P_\alpha(v)$ and $K_\alpha(v)$ are analytic functions on $\T^d\times \Sp^{d-1}.$
We assume that $\alpha$ is distributed according to a measure $\nu$ which has density $\psi.$ 
Let $\brlambda$ denote the product of $\nu$ and the normalized Lebesgue measure on 
$[a,b]\times \T^d.$ 

\begin{theo} The following limit holds. 
\label{ThParam}
$$\lim_{N \to \infty} \brlambda \{ (r, x, \a) \in [a,b] \times \T^d\times \T^d 
\Big| \frac{D_{\cC_\alpha} (r,\a,x,N)}{r^{\frac{d-1}{2}} N^{\frac{d-1}{2d}} } \leq z\}  = $$
$$\mu\times \nu\left\{ (L, (\theta, b), \a)\in \cM_d \times \T^d  : \cL_{\cC_\alpha}(L,\th, b) \leq z \right\}.$$
\end{theo}

\begin{proof}
The proof is similar to the proof of Theorem \ref{main.limit} so we only describe the necessary modifications.
Note that either for all $\alpha$,  $\cC_\alpha$ has a center of symmetry 
or the set of $\alpha$s such that $\cC_\alpha$ has a center of symmetry has measure 0. We consider the first
case the second case is similar. We also suppose that the centers of symmetry of all $\cC_\alpha$ are at the
origin (this can be always achieved by shifting $x$). Now the argument proceeds in the same way as the proof
of Theorem  \ref{main.limit} in the symmetric case except that Proposition \ref{prop.independence}
has to be straightened  as follows.

\begin{prop} 
\label{prop.independence2}
The random vectors
$$
\left((e_1(N,\a),\ldots, e_{d+1}(N, \a)), \a, 
\{\gamma_j\}_{j=1}^{d+1}, \quad
\{N^{\frac{1}{d}} P_\a(X_m) r \}_{m\in \cZ_\eps} \right)
$$
converge in distribution as $N\to\infty$ to $\mu_d\times \nu\times \lambda_{d, \eps}.$
\end{prop}
The proof of Proposition \ref{prop.independence2} proceeds in the same way as the
proof of Proposition \ref{prop.independence} except that \eqref{AsInd} has to be replaced by
\begin{equation}
\label{AsInd2} 
 \iiint  \Phi(e_1(N, \a),\dots, e_{d+1}(N, \a))
\end{equation}
$$ \times  \exp\left[ 2\pi i \left(\sum_{j=1}^{d+1} n_j \gamma_j +\sum_{\cZ_\eps} l_m A_m\right)\right] \psi(\a) 
dx d\a dr \to $$
$$ \int_{\T^d} \psi(\a)  d\a
\int_{M} \Phi(e_1(L),\dots e_{d+1}(L)) d\mu(L) $$
$$\times \int_{\T^{d+1}} e^{\sum_j n_j \gamma_j} d\gamma
\int_{\T^{\Z_\eps}} e^{\sum_m l_m A_m} dA \text{ as } N\to \infty.$$
To prove \eqref{AsInd2} note that the case when $n_j\equiv 0$ and $l_m\equiv 0$ reduces to 
\eqref{EquidWu2}. The case when some $n_j\neq 0$ is handled as in Proposition \ref{prop.independence}. Finally the case when 
$n_j \equiv 0$ but some $l_m \neq 0$ is similar to Proposition  Proposition \ref{prop.independence}  except that \eqref{AlmRes} now takes form 
\begin{equation}
\label{AlmRes22}
\mes\left(\a\in \T^d: \left|\sum_{i=1}^K l_i P_\alpha (X_{m_i}(L(N,\a)))\right|<\eps\right)\to 0 \text{ as } N\to\infty.
\end{equation}
To derive \eqref{AlmRes22} from \eqref{AlmRes} divide $\T^d$ into small cubes $\bC_s$ and for each $s$
pick $\a_s\in \bC_s.$ If the size of cubes is small enough then for $\a\in \bC_s$ the inequality
$$ \left|\sum_{i=1}^K l_i P_\alpha (X_{m_i}(L(N,\a)))\right|<\eps $$
holds provided that 
$$ \left|\sum_{i=1}^K l_i P_{\a_s} (X_{m_i}(L(N,\a)))\right|<\frac{\eps}{2}.  $$
Hence \eqref{AsInd2} follows from \eqref{AlmRes}.
\end{proof}

\subsection{Counting lattice points in slanted cylinders.}
\label{SSSlanted}
Given $v\in\R^{d+1},$ $r\in\R$ consider the cylinder 
\begin{equation}
\Cyl_{y,v,r, T}=\{z\in\R^{d+1}: |z-(y+tv)|<r \text{ for some } t\in [0,T]\}. 
\end{equation}
Let $N(y,v,rT)$ be the number of $\Z^{d+1}$ points in $\Cyl_{y,v,r,T}$ and
\begin{equation}
 \bbD(y,v,r,T)=N(y,v,r,T)-\Vol(\Cyl_{y,v,r, T}).
\end{equation}

We assume that $y=(x,0)$ and $v=(\a, 1)$ where $x, \a\in \R^d. $

\begin{theo}
\label{ThCyl}
If $b$ is sufficiently small then
$$\mathscr{D}(z)=\lim_{T \to \infty} \brlambda \{ (r, x, \a) \in [a,b] \times \T^d\times \T^d 
\Big| \frac{\bbD(y, v, r, T)}{r^{\frac{d-1}{2}} T^{\frac{d-1}{2d}} } \leq z\} $$
exists. An explicit formula for $\mathscr{D}$ is given by \eqref{BodyLat} and
\eqref{LimDistLat}.
\end{theo}

\begin{proof}
We are interested in the question under which condition the point
$\bm=(m_1, m_2,\dots, m_d, n)$ belongs to $\Cyl_{y,v,r,T}.$ Since edge effects contribute
$\cO(1)$ we may assume that $0\leq n\leq T.$ The plane $\{z_{d+1}=n\}$ intersects $\Cyl_{y,v,r,T}$
by an ellipsoid centered at $(x+n\alpha, n).$ Now elementary geometry shows that
$\bm\in \Cyl_{y,v,r,T}$ iff 
$$ (\alpha^2+1)|x_n-\brm|^2-(\a, x_n-\brm)^2\leq (\alpha^2+1) r^2$$
where $x_n=x-n\a,$ $\brm=(m_1,\dots, m_d).$  
The last condition can be restated by saying that $x+n\alpha$ mod $\Z^d$ belongs to
$r \cC_\alpha$ where 
\begin{equation}
\label{BodyLat}
\cC_\alpha=\{y \big| (\alpha^2+1)|y|^2-(\a, y)^2\leq (\alpha^2+1) \}
\end{equation}
Hence Theorem \ref{ThCyl} follows from Theorem \ref{ThParam} and
\begin{equation}
\label{LimDistLat}
\mathscr{D}(z)=\mu\times \nu\left\{ (L, (\theta, b), \a)\in \cM_d \times \T^d  : \cL_{\cC_\alpha}(L,\th, b) \leq z \right\}. \qedhere
\end{equation}
\end{proof}

\subsection{Proof of Theorem \ref{ThCont}(b) and Proposition \ref{prop.theorem3b}.}
\label{SSCont}
In this section we describe the proof of Theorem \ref{ThCont}(b). 
We only treat the case of a symmetric convex body. 
In Section \ref{SSRandGeo} we show that, in the case of balls, the limit distribution does not depend 
on the distribution $p$ of the translation vector.
 The argument is very similar to the proof of Theorem
\ref{ThParam} so we only give an outline of the proof.  
We have
\begin{equation}
\bD(r,v, x, T)=\sum_{k\in\Z^2-0} c_k \frac{\cos[2\pi(k,x)+\pi(k, Tv)] \sin(\pi(k, Tv))}{\pi (k,v)}
\end{equation}
where $c_k$ is given by formula \eqref{fourier.strictly.convex.symmetric}. Similarly to Section
\ref{ScNonRes} we show that it suffices to restrict our attention to the harmonics satisfying
\begin{equation} \eps<\frac{|k|}{T^{1/(d-1)}}<\eps^{-1}, \end{equation}
\begin{equation}
\label{ResCT}
\delta<T|(k,v)|<\delta^{-1}.
\end{equation}
Divide the support of $p$ onto small sets $\Omega_j$ such that on each $\Omega_j$, $v$ is almost constant.  
Fix one $\Omega_j$ and denote $\brv$ for an arbitrary choice of a point in $\Omega_j$.  Changing the indices if necessary we may assume that on $\Omega_j$, $v_d\neq 0$
so that we can write 
\begin{equation}
\label{Polar}
v=\rho (\alpha_1, \alpha_2\dots \alpha_{d-1}, 1), \quad \brv=\bar{\rho} (\bralpha_1, \bralpha_2\dots \bralpha_{d-1}, 1)
\end{equation}
Denote 
$M={\rm SL}(d,\R)/{\rm SL}(d,\Z)$. We let 
$$ g_n=\left(\begin{array}{cccc} e^{-n/(d-1)} & 0 & \ldots & 0 \cr
                                                            0 & e^{-n/(d-1)} & 0 & \ldots  \cr
                                                            & & & \cr
                                                            0 & \ldots & e^{-n/(d-1)}  & 0 \cr 
                                                            0 & \ldots & 0 & e^n  \cr
                                                             \end{array}\right), \quad
\Lambda_\alpha=\left(\begin{array}{cccc} 1 & 0 & 0 & \ldots \cr
                                                            0 & 1 & 0 & \ldots  \cr
                                                            \cr
						   \a_1 & \ldots & \a_{d-1} & 1 \cr
 \end{array}\right) . $$
Consider the lattice $L(T,\a)=g_{\ln T} \Lambda_\alpha \Z^{d}.$ Then
$$(X_1,\ldots,X_{d-1},Z) :=(k_1/T^{1/(d-1)},\ldots,k_{d-1}/T^{1/(d-1)},{T(k,\a)})=g_{\ln T} \Lambda_\alpha k$$
Due to \eqref{ResCT} we have
\begin{equation}
 \frac{k_d}{T^{1/(d-1)}}\approx -\sum_{s=1}^{d-1} \alpha_s X_s
\end{equation}
and hence
\begin{equation}
 |k|^{(d+1)/2}\approx T^{\frac{d+1}{2(d-1)}} \left[\sum_{s=1}^{d-1} X_s^2+\left(\sum_{s=1}^{d-1} \alpha_s X_s\right)^2
\right]^{\frac{d+1}{4}}. 
\end{equation}
The rest of the proof of Theorem \ref{ThCont}(b) proceeds similarly to the proof of Theorem \ref{main.limit}.  Namely, on $\Omega_j$ the distribution of $\bD(r,v,x,T)$ is approximated by the following distribution 
\begin{equation} \label{dist.dimd1} \sD_{\cC,\brv} (z) =  \mu\left\{ (L, (\theta, b))\in \cM_d : \sL_\brv(L,\th, b) \leq z \right\} \end{equation} 
where
\begin{multline} \nonumber \sL_{\brv}(L, \theta, b)= \\ \frac{2}{\pi^2} \sum_{m\in \cZ}\sum_{p=1}^\infty
  K^{-\frac{1}{2}}(X_m/R_m)  \frac{\cos(2\pi p(m,\theta)) \sin(2\pi p b_m) \sin (\pi p \bar{\rho} Z_m)}  
{  p^{\frac{d+3}{2}}  \bar{\rho} Q_m^{\frac{d+1}{2}} Z_m}. 
\end{multline}
Here $X_{m,s}$ is the $s$-th component of $X_m$, $R_m^2=\sum_{s=1}^{d-1} X_{m,s}^2$, and
$$Q_m^2=R_m^2+\left(\sum_{s=1}^{d-1} \bar{\alpha}_s X_{m,s} \right)^2.$$
By refining the division of the support of the distribution $p$ into smaller and smaller sets $\Omega_j$ we get that the limiting distribution  
$\bD(r,v,x,T)$ is given by $\int \sD_{\cC,v} (z) p(v) dv. 
\hfill \Box$



\subsection{Random geodesics on the torus.}
\label{SSRandGeo}
Let $\gamma_{x,v}(t)$ denote the geodesic $x+vt$ on $\T^d.$ Given $y, r$ let 
$\tau(r,v,x,y,T)$ denote the time $\gamma_{x,v}(t)$ spends inside $B(y,r)$ for $t\in [0,T]$.
Suppose that $y$ is fixed while $(r,v,x)$ are distributed according to 
the measure $\sigma$ as in Theorem \ref{ThCont}.

\begin{theo} \label{th8}
Suppose that $\rho<\frac{\sqrt{d}}{2}.$ Then

(a) If $d=2$ then the distribution of $\tau(r,v,x,y,T)-\Vol(B(y,r))T$ approaches a limit as $T\to\infty.$

(b)  If $d\geq 4$ then
\begin{equation}
\label{LimGeo}
\lim_{T\to\infty}
\sigma\left(\frac{|v|^{\frac{d+1}{2(d-1)}}}{r^{\frac{d-1}{2}}} \left( \frac{\tau(r,v,x,y,T)-\Vol(B(y,r))T}{T^{\frac{d-3}{2(d-1)}}}\right)
\leq z\right)=\sP(z) 
\end{equation} 
where 
$$\sP (z) =  \mu\left\{ (L, (\theta, b))\in \cM_d : \sL(L,\th, b) \leq z \right\}$$
and 
$$ \sL_{v}(L, \theta, b)=
\frac{2}{\pi^2} \sum_{m\in \cZ}\sum_{p=1}^\infty
 \frac{\cos(2\pi p(m,\theta)) \sin(2\pi p b_m) \sin (\pi p Z_m)}
{  p^{\frac{d+3}{2}}
R_m^{\frac{d+1}{2}} Z_m}. 
$$

\end{theo}
\begin{proof}
The existence of the limiting distribution follows immediately from 
Theorem \ref{ThCont}
(with $\cC=B(y,r)).$ It remains to show that the limit does not depend on the distribution of $v.$
Due to proposition \ref{prop.theorem3b} we have the following expression
for $\sP(z)$
$$ \sP(z)=\int \brsD_v(z) p(v) dv $$
where 
$$\brsD_{v} (z) =  \mu\left\{ (L, (\theta, b))\in \cM_d : \brsL_v(L,\th, b) \leq z \right\}.$$
Here $\brsL=|v|^{(d+1)/2(d-1)} \sL$ where 
$\sL_{v}$ is the same as in  \eqref{dist.dimd2} but specified to balls 
$$ \sL_{v}(L, \theta, b)= \frac{2}{\pi^2} \sum_{m\in \cZ}\sum_{p=1}^\infty
  \frac{\cos(2\pi p(m,\theta)) \sin(2\pi p b_m) \sin (\pi p \rho Z_m)}
{ p^{\frac{d+3}{2}} \rho Q_m^{\frac{d+1}{2}}  Z_m }. 
$$
Here $Q_m$ denotes
$$Q_m^2=\sum_{s=1}^{d-1} X_{m,s}^2+\left(\sum_{s=1}^{d-1} \alpha_s X_{m,s} \right)^2$$
and $X_{m,s}$ is the $s$-th component of $X_m.$
It remains to show that 
$\sD_{v}$ does not in fact depend on $v.$ We can choose coordinates in $\R^d$ so that
$\alpha_1=a,$ $\alpha_s=0$ for $s=2\dots d-1.$ Then
$$ Q_m^2=\frac{v^2}{\rho^2} X_{m,1}^2+X_{m,2}^2+\dots+X_{m, d-1}^2. $$
Note that the distribution of $\sL_v$ is invariant under unimodular linear transformations.
Therefore we can make the change of variables
$$ \brX_1=\frac{|v|}{\rho} \frac{X_1}{|v|^{1/(d-1)}}, \quad 
\brX_s=\frac{X_s}{|v|^{1/(d-1)}},  \text{ for } s=2\dots d-1, \quad \brZ=\rho Z. $$
Then
\begin{equation}
 \brsL_{v}(L, \theta, b)=
\frac{2}{\pi^2} \sum_{m\in \cZ}\sum_{p=1}^\infty
 \frac{\cos(2\pi p(m,\theta)) \sin(2\pi p b_m) \sin (\pi p \brZ_m)}
{  p^{\frac{d+3}{2}}
\brR_m^{\frac{d+1}{2}} \brZ_m}
\end{equation}
where
$$\brR_m^2=\sum_{s=1}^{d-1} \brX_{m,s}^2. $$
Since the RHS does not depend on $v$ the result follows.
\end{proof}

\section{Proof of Theorem \ref{ThCont} ($a$) and Proposition \ref{prop.theorem3a}}
\label{Contd2}
\begin{proof}
We have 
\begin{equation}
\label{CD2}
\bD_\cC(v, x, T)=\sum_{k\in\Z^2-0} c_k e^{2\pi i (k, y)} \frac{\sin(\pi(k, Tv))}{\pi (k,v)}
\end{equation}
where $y=x+\frac{vT}{2}$ and $c_k=\cO(|k|^{-3/2}).$ Note that for each $\eps$ for almost all $v$ there exists $n=n(v)$ such that
$|(k,v)|>|k|^{-1-\eps}$ for $|k|>n.$ Hence if $A_n =\left\{ v : |(k,v)|>|k|^{-1-\eps} \  Ê{ \rm Êfor Ê} \   |k|>n \right\}$, it holds that $  |A_n^c| \to 0$ as $n \to \infty$. Define 
\begin{equation}
\bD^n_+(v,x,T)=\sum_{|k|>n} c_k e^{2\pi i (k, y)}\frac{\sin(\pi(k, Tv))}{\pi (k,v)}, 
\end{equation}
\begin{equation}
\bD^n_-(v,x,T)=\sum_{|k|\leq n} c_k e^{2\pi i (k, y)}\frac{\sin(\pi(k, Tv))}{\pi (k,v)}.
\end{equation}

Let 
\begin{equation}
 A_{k,p}=\{v:     |(k,v)|\in [p |k|^{-1-\eps}, (p+1) |k|^{-1-\eps}] \}
 \end{equation}
then $ |A_{k,p}| \leq C |k|^{-2-\eps}$
and so 
\begin{equation}
 ||\bD^n_+ 1_{A_ n}||_{L^2(\bar \sigma)}^2\leq C\sum_{|k|>n} \sum_{p=1}^\infty
\frac{|k|^{\eps}}{ |k|^3 p^2}\leq C n^{-(1-\eps)}. 
\end{equation}
Accordingly, the distribution of $\bD_\cC$ is well approximated by the distribution of
$\bD^n_-$ if $n$ is large enough. On the other hand for each fixed $n$ the distribution 
of $\bD_-^n$ converges to a limit as $T\to\infty.$ Indeed remove a small neighborhood of
resonances and divide the remaining set into sets $\Omega_i$ of small diameter.
Then on each $\Omega_i$ the denominators in \eqref{CD2} are almost constant while
$\pi vT$ becomes uniformly distributed on $(\R/2\pi \Z)^2.$ 
Therefore the distribution of 
\begin{equation}
\label{LimDis2} 
\lim_{T\to\infty} \bar \sigma(\bD_\cC(v, x, T)<z)=
\int  \brsD_{\cC, v} (z) p(v) dv  
\end{equation}
where 
\begin{equation} 
\brsD_{\cC, v} (z)={\rm Leb} \left\{ (y,\theta) \in \T^2 \times \T^2 :  \cL_{v} (y,\theta) < z\right\}, 
\end{equation}
\begin{equation}
\label{LimDisFC}  
\cL_{v} (y,\theta) =\sum_{k\in\Z^2-0} c_k e^{2\pi i (k, y)} \frac{\sin(\pi(k, \theta))}{\pi (k,v)}  . 
\qedhere \end{equation}
\end{proof}

\begin{remark}
The fact that $c_k$ are Fourier coefficients of the indicator of $\cC$ is not important
in the above argument, only the rate of decay was used. Therefore the same argument 
gives the following statement
\begin{prop}
If $A$ belongs to a Sobolev space $\mathbb{H}^s$ for $s>\frac{d}{2}-1$ then
$$ \int_{0}^{T} A(S_v^t x) dt - T \int_{\T^d} A(x) dx$$
converges in distribution as $T\to\infty.$ The limiting
distribution is given by \eqref{LimDis2}--\eqref{LimDisFC} where $c_k$ are Fourier coefficients 
of $A.$
\end{prop}
\end{remark}


\section{Convergence of the function $\cL$} \label{L}
Here we prove that the series defining the limiting distributions in Propositions 
\ref{main3}, \ref{prop.theorem3a}, and \ref{prop.theorem3b} converge almost surely. 
\begin{prop}
The series \eqref{LatTor}, \eqref{LimSumSym}, \eqref{L.dim2},  \eqref{dist.dimd2}, and  \eqref{dist.dimd3} converge almost surely.
\end{prop}
\begin{proof}
We will prove the convergence of \eqref{LimSumSym}, the other series can be treated similarly.
Let 
\begin{equation}
\xi_m=\sum_p \frac{\sin(\pi p Z_m) \cos(2\pi p(\theta, m)) \sin(2\pi(p b_m-\frac{d-1}{8}))}{
R_m^{\frac{d+1}{2}} Z_m p^{\frac{d+3}{2}}} K^{-\frac{1}{2}}(X_m/R_m) .
\end{equation}
Note that for fixed $L$ and $\theta,$ the random variables $\xi_m$ are independent, and
\begin{equation} \E(\xi_m)=0, \quad \text{Var}(\xi_m)=\frac{\Gamma(\theta, Z_m)}{K(X_m/R_m) R_m^{d+1}} \end{equation}
where
\begin{equation} \Gamma(\theta, Z)=\sum_p \frac{\cos^2 (2\pi p(\theta, m)) \sin^2 (\pi p Z_m)}{Z_m^2 p^{d+3}}. \end{equation}
By Kolmogorov's three series theorem, given $(L, \theta)$, $\cL$ converges for almost 
every $b$ provided that
\begin{equation}
\label{SumVar}
\sum_m \frac{\Gamma(\theta, Z_m)}{R_m^{d+1}}<\infty.
\end{equation}
Therefore it suffices to show that \eqref{SumVar} converges for 
almost every $(L, \theta).$ 

Observe that
\begin{equation}
\label{PrRSmall}
 \Prob(R_m<s)=\cO\left(\frac{s^d}{|m|^d}\right)
\end{equation} 
so by Borel-Cantelli Lemma for each $\delta_0>0$ for almost every $L$ we have for sufficiently large $m$ that
$R_m>|m|^{-\frac{1+\delta_0}{d}}.$ 
Hence it is sufficient to show that
\begin{equation}
\label{SumVarCutOff1}
\sum_m \frac{\Gamma(\theta, Z_m)}{\brR_m^{d+1}}<\infty
\end{equation}
for almost every $(L, \theta)$ where $\brR_m=\max(|R_m|, |m|^{-\frac{1+\delta_0}{d}}).$ 
Note that if $|Z_m|>1$ then $\Gamma(\theta, Z_m)=\cO\left(|Z_m|^{-2}\right)$
and if $|Z_m|\leq1$ then $\Gamma(\theta, Z_m)=\cO\left(1\right).$
Accordingly it suffices to show that
\begin{equation}
\sum_m \frac{1}{\brZ_m^2 \brR_m^{d+1}}<\infty \end{equation}
for almost every $L$ where $\brZ_m=\max(|Z_m|,1).$

Next since every two norms on $\reals^{d+1}$ are equivalent there exist $c(L)$ such that
$|R_m|^2+Z_m^2\geq c|m|^2.$ Denote 
$\brrZ_m=\max(|m|^{1-\delta_0}, |Z_m|), $
$\brrR_m=\max(|m|^{1-\delta_0}, |R_m|). $
Then either $\brR_m=\brrR_m$ or $\brZ_m=\brrZ_m.$ 
Therefore it suffices to show that for almost all $L$

\begin{equation}
\label{SumVarCutOff2}
\sum_m \frac{1}{\brZ_m^2 \brrR_m^{d+1}}<\infty
\end{equation}
and

\begin{equation}
\label{SumVarCutOff3}
\sum_m \frac{1}{\brrZ_m^2 \brR_m^{d+1}}<\infty.
\end{equation}
To prove \eqref{SumVarCutOff2} it suffices to show that 
\begin{equation}
\label{SumVarCutOff4}
\sum_m \frac{1}{\brZ_m^2 |m|^{(1-\delta_0) (d+1)}}<\infty.
\end{equation}
Fix a compact set $\bbK$ in the space of lattices. Then there is a constant $C=C(\bbK)$ such  that 
\begin{equation}
 \Prob(Z_m\in [s,s+1] \text{ and }L\in \bbK) \leq \frac{C(\bbK)}{|m|}.
\end{equation}
Therefore 
\begin{equation} 
 \E\left(\frac{1}{\brZ_m^2} \chi_{L\in \bbK} \right)\leq \frac{\brC(\bbK)}{|m|}. 
\end{equation}
Now summation over $m$ shows that the series \eqref{SumVarCutOff4} converges for almost all $L\in \bbK.$ 
Since $\bbK$ is arbitrary \eqref{SumVarCutOff2} follows.

Likewise
\begin{equation} 
\label{ExpSumVarCutOff3}
\E\left(\frac{1}{\brrZ_m^2 \brR_m^{d+1}}\right)\leq \frac{\Const}{|m|^{2(1-\delta_0)}}  \E\left(\frac{1}{\brR_m^{d+1}}\right) . 
\end{equation}
Since \eqref{PrRSmall} implies that
\begin{equation*} 
\Prob \left( R_m \in [2^l |m|^{-(1+\delta_0)/d}, 2^{l+1} |m|^{-(1+\delta_0)/d}] \right)\leq \Const 2^{ld}  |m|^{-(1+\delta_0)-d}
\end{equation*}
we have
\begin{equation} 
\label{EXPRd}
 \E\left(\frac{1}{\brR_m^{d+1}}\right) \leq {\Const}{|m|^{ (1+\delta_0)/d -d}}
\end{equation} 
and \eqref{SumVarCutOff3} follows from \eqref{ExpSumVarCutOff3} and \eqref{EXPRd}.
\end{proof}




 \end{document}